\documentclass{ws-jktr}

\begin{document}

\markboth{N.~Abrosimov \& B.~Vuong}
{The volume of a compact hyperbolic antiprism}

%%%%%%%%%%%%%%%%%%%%% Publisher's Area please ignore %%%%%%%%%%%%%%
%\catchline{}{}{}{}{}
%%%%%%%%%%%%%%%%%%%%%%%%%%%%%%%%%%%%%%%%%%%%%%%%%%%%%%%%%%%%%%%%%%%

\title{The volume of a compact hyperbolic antiprism}

\author{Nikolay Abrosimov\footnote{Corresponding author.}}

\address{Regional Scientific and Educational Mathematical Center, \\
Tomsk State University, Tomsk, 634050, Russia\\[4pt]
Sobolev Institute of Mathematics, \\
Novosibirsk, 630090, Russia\\[4pt]
Novosibirsk State University,
Novosibirsk, 630090, Russia\\
abrosimov@math.nsc.ru}

\author{Bao Vuong}

\address{Sobolev Institute of Mathematics, \\
Novosibirsk, 630090, Russia\\[4pt]
Novosibirsk State University,
Novosibirsk, 630090, Russia\\
vuonghuubao@live.com}

\maketitle

\begin{abstract}
We consider a compact hyperbolic antiprism. It is a convex polyhedron with $2n$ vertices in the hyperbolic space $\mathbb{H}^3$. This polyhedron has a symmetry group $S_{2n}$ generated by a mirror-rotational symmetry of order $2n$, i.e. rotation to the angle $\pi/n$ followed by a reflection (see Fig.~1). 

We establish necessary and sufficient conditions for the existence of such polyhedra in $\mathbb{H}^3$. Then we find relations between their dihedral angles and edge lengths in the form of a cosine rule. Finally, we obtain exact integral formulas expressing the volume of a hyperbolic antiprism in terms of the edge lengths.
\end{abstract}



\section{Introduction}
An {\em antiprism} $\mathcal{A}_n$ is a convex polyhedron with two equal regular $n$-gons as the top and the bottom and $2n$ equal triangles as the lateral faces. The antiprism can be regarded as a drum with triangular sides (see Fig.~1 where for $n=5$ the lateral boundary is shown).
\begin{figure}[th]
\centerline{\psfig{file=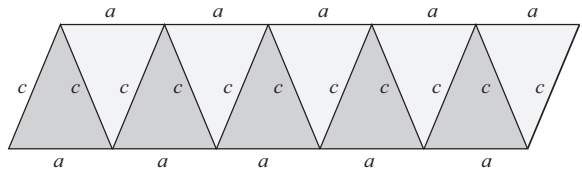,width=3in}}
\vspace*{8pt}
\caption{The lateral faces of antiprism $\mathcal{A}_5$.\label{fig1}}
\end{figure}

An antiprism $\mathcal{A}_n$ with $2n$ vertices has a symmetry group $S_{2n}$ generated by a mirror-rotational symmetry of order $2n$ denoted by $C_{2n\,h}$ (in Sh\"onflies notation). In Hermann--Mauguin notation this type of symmetry is denoted by $\overline{2n}$. The element $C_{2n\,h}$ is a composition of a rotation by the angle of $\pi/n$ about an axis passing through the centres of the top and the bottom faces and reflection with respect to a plane perpendicular to this axis and passing through the middles of the lateral edges (see Fig.~2).
\begin{figure}[th]
\centerline{\psfig{file=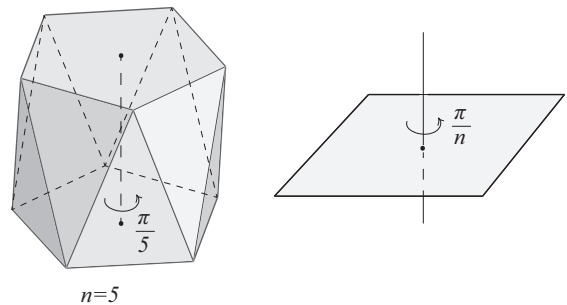,width=2.9in}}
\vspace*{8pt}
\caption{The symmetry of an antiprism.\label{fig2}}
\end{figure}

The above definitions of an antiprism $\mathcal{A}_n$ and its symmetry group $S_{2n}$ take place either for Euclidean or the hyperbolic space. By definition, $\mathcal{A}_n$ has two types of edges. Denote by $a$ the length of those edges that form top and bottom $n$-gonal faces. Set $c$ for the length of the lateral edges. Denote the dihedral angles by $A, C$ respectively. From here on, we will designate as $\mathcal{A}_n(a,c)$ for the antiprism $\mathcal{A}_n$ given by its edge lengths $a,c$.

The ideal antiprism in $\mathbb{H}^3$ with all vertices at infinity was studied by A.~Yu.~Vesnin and A.~D.~Mednykh \cite{VesMed1995} (see also \cite{Ves1996}). A particular case of ideal rectangular antiprism is due to W.~P.~Thurston \cite{Thu1980}. In the case of ideal antiprism the dihedral angles are related by a condition $2A+2C=2\pi$ while in a compact case the inequality $2A+2C>2\pi$ holds.

For $n=2$ the $n$-gons at the top and the bottom of antiprism $\mathcal{A}_n$ degenerate to corresponding two skew edges. Thus we obtain a tetrahedron with symmetry group $S_4$ (see Fig.~3). The volume a compact hyperbolic tetrahedron of this type was given by the authors in \cite{AbrVuo2017}.

\begin{figure}[th]
\centerline{\psfig{file=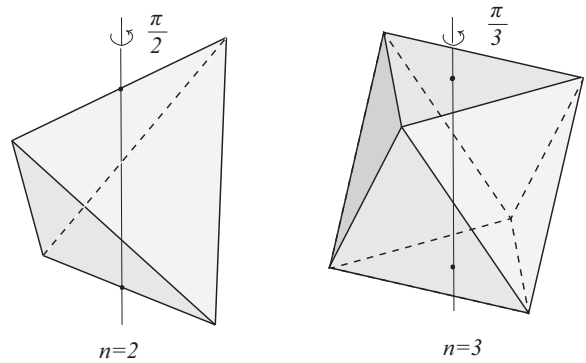,width=3in}}
\vspace*{8pt}
\caption{Particular antiprisms $\mathcal{A}_2$ and $\mathcal{A}_3$.\label{fig3}}
\end{figure}

For $n=3$ the antiprism $\mathcal{A}_n$ is an octahedron with symmetry group $S_6$ (see Fig.~3). The volume of a compact hyperbolic octahedron with this type of symmetry was found by the first author, E.~S.~Kudina and A.~D.~Mednykh in \cite{AbrKudMed2015}.

In the present work we consider a general case of a compact antiprism $\mathcal{A}_n$ in $\mathbb{H}^3$.

\section{Euclidean Antiprism}
Consider a rectangular coordinate system $Ox_1x_2x_3$ in the $3$-dimensional Eucli\-dean space $\mathbb{E}^3$ with the scalar product $\langle\cdot,\cdot\rangle_{\mathbb{E}}$. A mirror-rotational symmetry of order $2n$ is a rotation to the angle $\pi/n$ about the $Ox_3$ axis followed by a reflection with respect to the coordinate plane $Ox_1x_2$. It is defined by the matrix
$$
C_{2n\,h}=\left(
     \begin{array}{ccc}
       \cos \dfrac {\pi}{n} & -\sin \dfrac {\pi}{n} & ~~0\\[7pt]
       \sin \dfrac {\pi}{n} & ~~\cos \dfrac {\pi}{n} & ~~0\\[7pt]
       0 & 0 & -1\\
     \end{array}
   \right).
$$

Since the symmetry group acts transitively on the set of vertices, it suffices to point out the coordinates of one of the vertices. Without loss of generality, we can assume that a vertex $v_1$ has coordinates $(r,0,h/2)$ with some positive real numbers $r$ and $h$. The orbit of $v_1$ under the action of $C_{2n\,h}$ consists of all vertices of the antiprism $\mathcal{A}_n$.
$$
C_{2n\,h}\,:\,v_i \longmapsto v_{i+1},%\quad (i\in\mathbb{Z}_{2n})
$$
where the indices are taken modulo $2n$. 

Let us write the coordinates of the vertices of the antiprism $\mathcal{A}_n$ in $\mathbb{E}^3$
\noindent
\begin{equation}\begin{split}\label{EucCoor}
v_{2k+1}&=\left(r\,\cos\frac{2k\pi}{n}, r\,\sin\frac{2k\pi}{n}, \frac{h}{2}\right),\\
v_{2k+2}&=\left(r\,\cos\frac{(2k+1)\pi}{n}, r\,\sin\frac{(2k+1)\pi}{n}, -\frac{h}{2}\right),
\end{split}\end{equation}
%\begin{align}
%v_{2k+1}&=\left(r\,\cos\frac{2k\pi}{n}, r\,\sin\frac{2k\pi}{n}, \frac{h}{2}\right)\\
%v_{2k+2}&=\left(r\,\cos\frac{(2k+1)\pi}{n}, r\,\sin\frac{(2k+1)\pi}{n}, -\frac{h}{2}\right),
%\end{align}
where $k=0,\ldots,n-1$. Odd vertices form the top $n$-gonal face and even vertices form the bottom $n$-gonal face of $\mathcal{A}_n$.

The squared edge lengths of the antiprism are as follows
\noindent
\begin{equation*}
a^2=4\,r^2\sin^2\frac{\pi}{n},\quad c^2=h^2+4\,r^2\sin^2\frac{\pi}{2n}\,.
\end{equation*}

We can express the parameters $r,h$ from the latter equations
\noindent
\begin{equation}\label{rhExp}
r^2=\frac{a^2}{4\sin^2\frac{\pi}{n}},\quad h^2=c^2-\frac{a^2}{4\cos^2\frac{\pi}{2n}}\,.
\end{equation}

Using the coordinates of the vertices we calculate the cosines of the dihedral angles by the inner products of the outward normal vectors of corresponding faces. Then replace the parameters $r,h$ by the above expressions in terms of edge lengths. We obtain
\noindent
\begin{equation*}
\cos A=\frac{-a\tan\frac{\pi}{2n}}{\sqrt{4\,c^2-a^2}},\quad \cos C=\frac{a^2-4\,c^2\cos\frac{\pi}{n}}{4\,c^2-a^2}\,.
\end{equation*}

The condition $h^2>0$ or, equivalently, 
\begin{equation}\label{EucExist}
4\,c^2\cos\frac{\pi}{2n}-a^2>0
\end{equation} 
is necessary and sufficient for the existence of an antiprism $\mathcal{A}_n(a,c)$ in $\mathbb{E}^3$. The inequality $4\,c^2-a^2>0$ follows from (\ref{EucExist}). The dihedral angle $A$ is always greater than $\pi/2$.
If $4\,c^2\cos\frac{\pi}{2n}-a^2=0$ then the height $h$ vanishes and the antiprism degenerates to the planar regular $2n$-gone. 

Let us divide an Euclidean antiprism $\mathcal{A}_n$ into $n$ equal parts as shown in Fig.~4. Each of this parts consists of three tetrahedra $T_1, T_2, T_3$. Consider the tetrahedra $T_1=BDv_2v_4$ and $T_2=Bv_1v_2v_3$. They are not congruent but have the same volume
$$
vol(T_1)=vol(T_2)=\dfrac{a^2h}{12}\cot\dfrac{\pi}{n}.
$$

\begin{figure}[th]
\centerline{\psfig{file=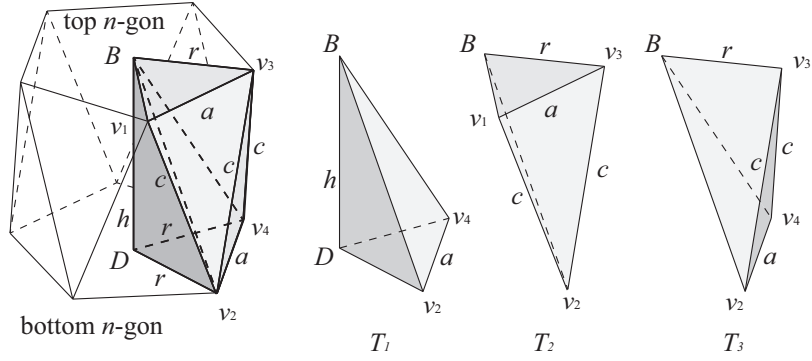,width=4.2in}}
\vspace*{8pt}
\caption{Decomposition of an antiprism $\mathcal{A}_n$.\label{fig4}}
\end{figure}

We calculate the volume of the tetrahedron $T_3=Bv_3v_2v_4$ by Servois' formula (see, e.g., \cite{Pon}, p.~98)
$$
vol(T_3)=\dfrac{r\,a\,h}{6}\sin\dfrac{\pi}{2}.
$$

Then we substitute $r,h$ using the expressions (\ref{rhExp}). Thus, the volume of Euclidean antiprism $\mathcal{A}_n(a,c)$ is
\begin{equation*}
vol(\mathcal{A}_n(a,c))=\dfrac{n\,a^2(2\cos\frac{\pi}{n}+1)}{24\sin\frac{\pi}{n}\cos\frac{\pi}{2n}}\sqrt{4\,c^2\cos^2\frac{\pi}{2n}-a^2}.
\end{equation*}

In particular, for the case $n=3$ we have
\begin{align*}
a^2&=3\,r^2,& \cos A&=\frac{-r}{\sqrt{4\,h^2+r^2}}=\frac{-a}{\sqrt{12\,c^2-3\,a^2}},\\
c^2&=h^2+r^2,& \cos C&=\frac{r^2-2\,h^2}{4\,h^2+r^2}=\frac{a^2-2\,c^2}{4\,c^2-a^2},
\end{align*}
and the volume is 
$$
vol(\mathcal{A}_3(a,c))=\dfrac{a^2}{3}\sqrt{3\,c^2-a^2}
$$
which is the same as in \cite{AbrKudMed2015} and coincide with the result of \cite{GMS}.

\section{Compact Hyperbolic Antiprism}
\subsection{Projective Caley--Klein model}
Consider the Minkowski space $\mathbb{R}^4_1$ with inner product $\langle X,Y\rangle=-\,x_1y_1-x_2y_2-x_3y_3+x_4y_4$. The {\em Cayley--Klein model} of the hyperbolic space $\mathbb{H}^3$ is the set of vectors that form the unit ball $K=\{(x_1,x_2,x_3,1)\;:\;x_1^2+x_2^2+x_3^2<1\}$ lying in the hyperplane $x_4=1$. Straight lines and planes in this model are given by the intersections of the ball $K$ with Euclidean lines and planes lying in the hyperplane $x_4=1$.

Let $V$ and $W$ be two vectors in $K$. Set $V=(v,1)$ and $W=(w,1)$, where $v,w\in\mathbb{R}^3$. Then the inner product of these vectors in the Minkowski space and the Euclidean inner product of $v$ and $w$ are related as $\langle V,W\rangle=1-\langle v, w\rangle_\mathbb{E}$.

The {\em distance} $\rho(V,W)$ between vectors $V$ and $W$ in the Cayle--Klein model is defined by the equality
\begin{equation}\label{distance}
\cosh\rho(V,W)=\dfrac{\langle V,W\rangle}{\sqrt{\langle V,V\rangle\langle W,W\rangle}}.
\end{equation}

A {\em plane} in the model $K$ can be defined as the locus $P=\{V\in K:\langle V,N\rangle=0\}$, where $N=(n,1), \langle n,n\rangle_\mathbb{E}>0$, is a {\em normal vector} to the plane $P$. The point $(n,1)$ is called a {\em pole} of $P$ and situated outside $K$.

In the Cayley--Klein model, consider an inner dihedral angle $\theta$ formed by two planes $P,Q$. We denote by $N,M$ the normal vectors to the planes $P,Q$ directed outwards of the dihedral angle $\theta$. Then
%In the Cayley--Klein model, consider planes $P,Q$ with normals $N,M,$ respectively. Each of the four {\em dihedral angles} between the planes $P$ and $Q$ is defined by the relation
\begin{equation}\label{DihAngleFormulaHyp}
\cos\theta=-\dfrac{\langle N,M\rangle}{\sqrt{\langle N,N\rangle\langle M,M\rangle}}.
\end{equation}

Let $V_1=(v_1,1), V_2=(v_2,1), V_3=(v_3,1)$ be three noncoplanar vectors in $K$. Then there exists a unique plane $P=\{V\in K:\langle V,N\rangle=0\}$ passing through these vectors. The coordinates of the component $n$ of its normal vector $N=(n,1)$ are uniquely determined as a solution to the system of linear equations
\begin{equation}\begin{split}\label{nor}
\langle v_1,n\rangle_\mathbb{E}-1&=0\,,\\
\langle v_2,n\rangle_\mathbb{E}-1&=0\,,\\
\langle v_3,n\rangle_\mathbb{E}-1&=0\,.
\end{split}\end{equation}

\subsection{Edge lengths and existence condition}
Consider a compact hyperbolic antiprism $\mathcal{A}_n(a,c)$ given by its hyperbolic edge lengths $a,c$. We denote by $a$ the length of those edges that form top and bottom $n$-gonal faces. We set $c$ for the length of the lateral edges.

\begin{theorem}\label{existe}
A compact hyperbolic antiprism $\mathcal{A}_n(a,c)$ with the symmetry group $S_{2n}$ is exist if and only if 
\begin{equation*}
1+\cosh a-2\cosh c+2\,(1-\cosh c)\cos\frac{\pi}{n}<0.
\end{equation*}
\end{theorem}

\begin{proof}
A mirror-rotational symmetry $C_{2n\,h}$ in $\mathbb{E}^3$ can be naturally extended to the symmetry of the same type in $\mathbb{H}^3$. Consider a projective Caley--Klein model of $\mathbb{H}^3$ with a rectangular coordinate system $Ox_1x_2x_3x_4$. As in Euclidean space, a mirror-rotational symmetry of order $2n$ is a rotation to the angle $\pi/n$ about the $Ox_3$ axis followed by a reflection with respect to the coordinate plane $Ox_1x_2$. It is defined by the matrix 
$$
\overline{C_{2n\,h}}=\left(
     \begin{array}{cccc}
       \cos \dfrac {\pi}{n} & -\sin \dfrac {\pi}{n} & ~~0 & ~~~~~0\\[7pt]
       \sin \dfrac {\pi}{n} & ~~\cos \dfrac {\pi}{n} & ~~0 & ~~~~~0\\[7pt]
       0 & 0 & -1 & ~~~~~0\\[7pt]
       0 & 0 & ~~0 & ~~~~~1\\
     \end{array}
   \right).
$$

We place an Euclidean antiprism into the ball $K=\{(x_1,x_2,x_3,1):x_1^2+x_2^2+x_3^2<1\}$. For this, we add the coordinate $x_4=1$ to each of the vertices $v_i$ in (\ref{EucCoor})
\begin{equation}\begin{split}\label{HypCoor}
V_{2k+1}&=\left(r\,\cos\frac{2k\pi}{n}, r\,\sin\frac{2k\pi}{n}, \frac{h}{2},1\right),\\
V_{2k+2}&=\left(r\,\cos\frac{(2k+1)\pi}{n}, r\,\sin\frac{(2k+1)\pi}{n}, -\frac{h}{2},1\right),
\end{split}\end{equation}
where $k=0,\ldots,n-1$. An additional condition 
\begin{equation}\label{inside}
r^2+\dfrac{h^2}{4}<1
\end{equation}
must be satisfied. It ensures that all vertices $V_i$ belong to the unit ball $K$.
The symmetry group acts transitively on the set of vertices
$$
\overline{C_{2n\,h}}\,:\,V_i \longmapsto V_{i+1},%\quad (i\in\mathbb{Z}_{2n})
$$
where the indices are taken modulo $2n$. By definition, given vertices $V_i$ form an antiprism in $\mathbb{H}^3$.

Knowing the coordinates of vertices, we calculate the edge lengths of a hyperbolic antiprism $\mathcal{A}_n$ by the formula (\ref{distance})
\begin{equation}\label{HypLeng}
\cosh a=\dfrac{4-h^2-4r^2\cos\frac{2\pi}{n}}{4-h^2-4r^2},\quad \cosh c=\dfrac{4+h^2-4r^2\cos\frac{2\pi}{n}}{4-h^2-4r^2}.
\end{equation}

The conditions \,$\cosh a,\,\cosh c \geq 1$\, follows immediately from (\ref{HypLeng}) and (\ref{inside}). Solving the system of equations (\ref{HypLeng}) with respect to $r^2$ and $h^2$, we obtain the following useful relations
\begin{equation}\begin{split}\label{rhExpHyp}
r^2&=\dfrac{\cosh a-1}{(1+\cosh a+2\cosh c-2(1+\cosh c)\cos\frac{\pi}{n})\cos^2\frac{\pi}{2n}},\\
h^2&=-\dfrac{4(1+\cosh a-2\cosh c-2(\cosh c-1)\cos\frac{\pi}{n})\tan^2\frac{\pi}{2n}}{1+\cosh a+2\cosh c-2(1+\cosh c)\cos\frac{\pi}{n}}.
\end{split}\end{equation}

Since \,$1+\cosh a+2\cosh c-2(1+\cosh c)\cos\frac{\pi}{n}>0$\, for any values of $a,c>0$, then the expression for $r^2$ is always positive. But from the expression for $h^2$ we conclude that the edge lengths $a,c$ can not be chosen arbitrarily. They should satisfy the condition
\begin{equation}\label{assump}
1+\cosh a-2\cosh c-2(\cosh c-1)\cos\frac{\pi}{n}<0.
\end{equation}
%which is the same as in the statement of the theorem. 

In the limit case of \,$1+\cosh a-2\cosh c-2(\cosh c-1)\cos\frac{\pi}{n}=0$, the hyperbolic antiprism $\mathcal{A}_n(a,c)$ degenerates into a plane $2n$-gon. This can be imagined by taking the parameter $h$ in Fig.~4 equal to zero.

We note that under the assumption (\ref{assump}), the inequality (\ref{inside}) follows from relations (\ref{rhExpHyp}). Thus, the inequality (\ref{assump}) is a necessary and sufficient condition for the existence of a compact antiprism $\mathcal{A}_n(a,c)$ in $\mathbb{H}^3$.
\end{proof}

For the case $n=3$, an antiprism $\mathcal{A}_3(a,c)$ is an octahedron $\mathcal{O}(a,c)$ with the symmetry group $S_6$. From Theorem~\ref{existe} we have the following.

\begin{corollary}
A compact hyperbolic octahedron $\mathcal{O}(a,c)$ with the symmetry group $S_6$ and edge lengths $a,c$ is exist if and only if 
\begin{equation*}
2+\cosh a-3\cosh c<0.
\end{equation*}
\end{corollary}
This coincides with the result of \cite{AbrKudMed2015} (Remark~4.1).

\subsection{Dihedral angles}
\begin{theorem}\label{ThDihAngHyp}
Let $\mathcal{A}_n(a,c)$ be a compact hyperbolic antiprism with $2n$ vertices given by its edge lengths $a,c$. Then the dihedral angles of $\mathcal{A}_n(a,c)$ can be found by the formulas
\begin{equation}\begin{split}\label{HypDihCos}
\cos A&=\frac{-\sqrt{\cosh a-1}\left(1+\cosh a-2\cosh c\,\cos\frac{\pi}{n}\right)}{\sqrt{2(1+\cosh a-2\cosh^2 c)(\cos\frac{2\pi}{n}-\cosh a)}},\\
\cos C&=\frac{\cosh c-\cosh a\,\cosh c+2(\cosh^2 c-1)\cos\frac{\pi}{n}}{1+\cosh a-2\cosh^2 c}.
\end{split}\end{equation}
\end{theorem}

\begin{proof}
We calculate the normal vectors to the faces $BV_1V_3, V_1V_2V_3$, and $V_2V_3V_4$ of the antiprism $\mathcal{A}_n(a,c)$ (we use the same order of vertices as in Fig.~4). To do this, we take the coordinates of vertices (\ref{HypCoor}) and solve the corresponding system of equations (\ref{nor}) for each of this faces. We get
\begin{equation*}\begin{split}
N_{BV_1V_3}&=\left(0,0,\frac{2}{h},1\right),\\
N_{V_1V_2V_3}&=\left(\frac{\cos\frac{\pi}{n}}{r\cos^2\frac{\pi}{2n}},\frac{2\tan\frac{\pi}{2n}}{r},\frac{2\tan^2\frac{\pi}{2n}}{h},1\right),\\
N_{V_2V_3V_4}&=\left(\frac{\cos\frac{2\pi}{n}}{r\cos^2\frac{\pi}{2n}},\frac{4\cos\frac{\pi}{n}\tan\frac{\pi}{2n}}{r},-\frac{2\tan^2\frac{\pi}{2n}}{h},1\right).
\end{split}\end{equation*}

One can ascertain that the obtained normal vectors are directed outwards of the antiprism. Then we use formula (\ref{DihAngleFormulaHyp}) to calculate the dihedral angles of $\mathcal{A}_n(a,c)$
\begin{equation*}\begin{split}
\cos A&=\frac{r\left(h^2-4\tan^2\frac{\pi}{2n}\right)}{\sqrt{(h^2-4)\left(h^2(r^2-1)-2h^2\tan^2\frac{\pi}{2n}-(h^2+4r^2)\tan^4\frac{\pi}{2n}\right)}},\\
\cos C&=\frac{4(h^2(r^2-2)-4r^2)\cos\frac{\pi}{n}+r^2(4+h^2)\left(3+\cos\frac{2\pi}{n}\right)}{8h^2-r^2\left(4(4+h^2)\cos\frac{\pi}{n}+(h^2-4)\left(3+\cos\frac{2\pi}{n}\right)\right)}.
\end{split}\end{equation*}

It follows from (\ref{inside}) that the expression under the square root is always positive.

Finally, we substitute parameters $r,h$ in the latter relations with the expressions (\ref{rhExpHyp}) to get the sought formulas (\ref{HypDihCos}).
\end{proof}

For $n=3$, an antiprism $\mathcal{A}_3(a,c)$ is an octahedron $\mathcal{O}(a,c)$ with the symmetry group $S_6$. From Theorem~\ref{ThDihAngHyp} we have the following.

\begin{corollary}
Let $\mathcal{O}(a,c)$ be a compact hyperbolic octahedron with the symmetry group $S_6$ and edge lengths $a,c$. Then the dihedral angles of $\mathcal{O}(a,c)$ can be found by the formulas 
\begin{equation*}\begin{split}
\cos A&=\frac{-\sqrt{\cosh a-1}(1+\cosh a-\cosh c)}{\sqrt{(2\cosh^2 c-\cosh a-1)(1+2\cosh a)}},\\
\cos C&=\frac{1-\cosh c+\cosh a\,\cosh c-\cosh^2 c}{2\cosh^2 c-\cosh a-1}.
\end{split}\end{equation*}
\end{corollary}
This coincides with the result of \cite{AbrKudMed2015} (Formulas~(4.13)).

\subsection{Volume formula}
\begin{theorem}\label{mainTheo}
Let $\mathcal{A}_n(a,c)$ be a compact hyperbolic antiprism with $2n$ vertices given by its edge lengths $a,c$. Then the volume $V=vol(\mathcal{A}_n(a,c))$ can be found by the formula
\begin{equation}\label{VolumeFormula}
V=n \int_{c_0}^c \frac{a\,G+t\,H}{(2\cosh^2 t-1-\cosh a)\sqrt{R}}\,dt,
\end{equation}
where
\begin{flalign*}
G&=2\left(\cosh t-\cos\frac{\pi}{n}\right)\sinh a\;\sinh t,\\
H&=-(\cosh a-1)\left(1+\cosh a+2\cosh^2 t-4\cosh t\cos\frac{\pi}{n}\right),\\
R&=1-\cosh a\,(2+\cosh a)+2\cosh^2 t+4\,(\cosh a-1)\cosh t\cos\frac{\pi}{n}-2\sinh^2 t\cos\frac{2\pi}{n}
\end{flalign*}
and $c_0$ is the root of the equation $2\cosh c\left(1+\cos\frac{\pi}{n}\right)=1+\cosh a+2\cos\frac{\pi}{n}$.
\end{theorem}

\begin{proof}
Consider a compact hyperbolic antiprism $\mathcal{A}_n(a,c)$ with $2n$ vertices given by its edge lengths $a,c$. By definition, $\mathcal{A}_n(a,c)$ has the symmetry group $S_{2n}$. According to Theorem~\ref{existe}, the domain of existence for such an antiprism has the form
$$
\Omega=\{(\cosh a,\cosh c)\,:\,\cosh a>1, \;2\cosh c\left(1+\cos\frac{\pi}{n}\right)>1+\cosh a+2\cos\frac{\pi}{n}\},
$$
in the coordinate system $(\cosh a,\cosh c)$ (see Fig.~5).
\begin{figure}[h]
\begin{center}
  \includegraphics[scale=0.4]{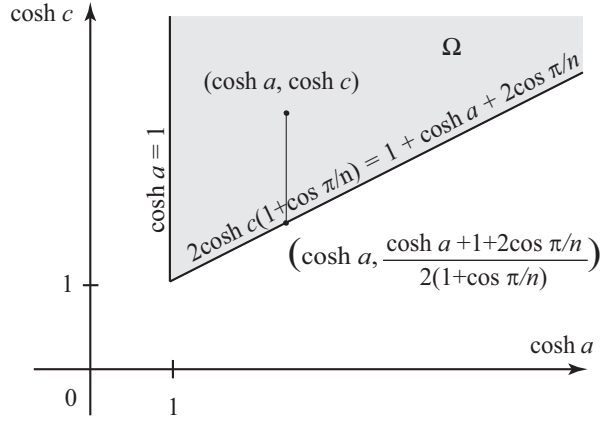}
\end{center}
\caption{Domain of existence for a hyperbolic antiprism $\mathcal{A}_n(a,c)$}
\end{figure}

The boundary of the domain $\Omega$ consists of two rays $\{\cosh a=1,\cosh c\geq 1\}$ and $\{\cosh a\geq 1, \;2\cosh c\left(1+\cos\frac{\pi}{n}\right)=1+\cosh a+2\cos\frac{\pi}{n}\}$. In each of them the antiprism $\mathcal{A}_n(a,c)$ loses dimension by degenerating into a line segment or plane $2n$-gon correspondingly. Hence, the equality $V=0$ holds on the boundary of $\Omega$.

Denote by $A,C$ the dihedral angles along the edges $a,c$ of $\mathcal{A}_n(a,c)$ respectively. According to Theorem~\ref{ThDihAngHyp}, the dihedral angles are uniquely determined by the edge lengths. We differentiate the volume as a composite function on the edge lengths
\begin{equation}\label{V1}
\frac{\partial V}{\partial c}=\frac{\partial V}{\partial A}\frac{\partial A}{\partial c}+\frac{\partial V}{\partial C}\frac{\partial C}{\partial c}.
\end{equation}

By the Schl\"afli formula (see, e.g., \cite{Vinberg}, Ch.~7, Sect.~2.2), we have
\begin{equation*}
d V=-\sum_{\theta}\frac{\ell_{\theta}}{2}\,d\theta=-n\,a\,dA\,-n\,c\,dC,
\end{equation*}
where the sum is taken over all edges of $\mathcal{A}_n(a,c), \;\ell_{\theta}$ denotes the edge length and $\theta$ is the interior dihedral angle along it. Consequently, 
\begin{equation}\label{V2}
\dfrac{\partial V}{\partial A}=-n\,a,\quad\dfrac{\partial V}{\partial C}=-n\,c.
\end{equation} 

From (\ref {HypDihCos}) by straightforward calculation we obtain
\begin{equation}\begin{split}\label{V3}
\frac{\partial A}{\partial c}&=\frac{2\left(\cosh c-\cos\frac{\pi}{n}\right)\sinh a\;\sinh c}{(1+\cosh a-2\cosh^2 c)\sqrt{R}},\\
\frac{\partial C}{\partial c}&=\frac{-(\cosh a-1)(1+\cosh a+2\cosh^2 c-4\cosh c\cos\frac{\pi}{n})}{(1+\cosh a-2\cosh^2 c)\sqrt{R}},
\end{split}\end{equation} 
where 
$$R=1-\cosh a\,(2+\cosh a)+2\cosh^2 c+4\,(\cosh a-1)\cosh c\cos\frac{\pi}{n}-2\sinh^2 c\,\cos\frac{2\pi}{n}.$$

We substitute (\ref{V2}) and (\ref{V3}) in (\ref{V1}). Then we have 
\begin{equation}\label{dVc}
\frac{\partial V}{\partial c}=n\frac{a\,G+c\,H}{(2\cosh^2 c-\cosh a-1)\sqrt{R}},
\end{equation}
where
\begin{flalign*}
G&=2\left(\cosh c-\cos\frac{\pi}{n}\right)\sinh a\;\sinh c,&\\
H&=-(\cosh a-1)\left(1+\cosh a+2\cosh^2 c-4\cosh c\cos\frac{\pi}{n}\right).&
\end{flalign*}

The integral of the differential form 
\begin{equation}\label{dV}
d V=\dfrac{\partial V}{\partial a}da+\dfrac{\partial V}{\partial c}dc
\end{equation} 
does not depend on the path of integration connecting two fixed points in $\Omega$. Since the volume vanishes in the boundary of $\Omega$ then by the Newton--Leibniz formula, the volume is equal to the integral of the form (\ref{dV}) along any piecewise smooth path $\gamma\subset\Omega$ beginning from any point at the boundary of $\Omega$ with the end at the point $(\cosh a,\cosh c)$. We substitute the expression (\ref{dVc}) in (\ref{dV}). Then we integrate the differential form (\ref{dV}) over the vertical segment shown in Fig.~5 which connects the boundary of $\Omega$ with the point $(\cosh a,\cosh c)$. Thus, we arrive at the formula (\ref{VolumeFormula}). To distinguish the edge length $c$ from the variable of integration we denote the variable by $t$.
\end{proof}

For $n=3$, an antiprism $\mathcal{A}_3(a,c)$ is an octahedron $\mathcal{O}(a,c)$ with the symmetry group $S_6$. From Theorem~\ref{mainTheo} we have the following.

\begin{corollary}
Let $\mathcal{O}(a,c)$ be a compact hyperbolic octahedron with the symmetry group $S_6$ given by its edge lengths $a,c$. Then the volume can be found by the formula
\begin{multline*}
vol(\mathcal{O}(a,c))=\\
3 \int_{c_0}^c \frac{a(2\cosh t-1)\sinh a\,\sinh t-t(\cosh a-1)(1+\cosh a+2\cosh t(\cosh t-1))}{(\cosh 2t-\cosh a)\sqrt{(3\cosh c-\cosh a-2)(\cosh a+\cosh c)}}\,dt,
\end{multline*}
where $c_0={\rm arccosh}\dfrac{\cosh a+2}{3}$.
\end{corollary}

This completely coincides with the result of \cite{AbrKudMed2015} (Formula~(4.18)).

\appendix

\section*{Acknowledgments}
This work was supported by the Ministry of Education and Science of Russia (state assignment No. 1.12877.2018/12.1).

\end{document}